\documentclass[12pt]{amsart}

\theoremstyle{plain}
\newtheorem{thm}{Theorem}[section]

\newtheorem{lem}[thm]{Lemma}
\newtheorem{pro}[thm]{Proposition}

\theoremstyle{definition}
\newtheorem{rem}[thm]{Remark}
\newtheorem{defin}[thm]{Definition}

\newtheorem*{2loc}{Theorem 4.4}
\newtheorem*{2glob}{Theorem 4.9}
\newtheorem*{conj}{Conjecture}

\begin{document}

\title[Sections of Serre fibrations with 2-manifold fibers]
{Sections of Serre fibrations with 2-manifold fibers}

\author{N. Brodsky}
\address{University of Tennessee, Knoxville, TN 37996, USA}
\email{brodskiy@math.utk.edu}

\author{A. Chigogidze}
\address{University of North Carolina at Greensboro, Greensboro, NC
27402, USA} \email{chigogidze@uncg.edu}

\author{E.V.~\v{S}\v{c}epin}
\address{Steklov Institute of Mathematics,
Russian Academy of Sciences, Moscow 117966, Russia}
\email{scepin@mi.ras.ru}

\keywords{Serre fibration; section; selection; approximation.}
\subjclass{Primary: 57N05, 57N10; Secondary: 54C65.}

\begin{abstract}
It was proved by H. Whitney in 1933 that a Serre
fibration of compact metric spaces admits a global section
provided every fiber is homeomorphic to the unit interval [0,1].
Results of this paper extend Whitney theorem to the
case when all fibers are homeomorphic to a given compact
two-dimensional manifold.
\end{abstract}

\maketitle

\section{Introduction}\label{S:intro}

The following problem is one of the central problems in geometric
topology~\cite{DS}. Let $p\colon E\to B$ be a Serre fibration of
separable metric spaces. Suppose that the space $B$ is locally
$n$-connected and all fibers of $p$ are homeomorphic to a given
$n$-dimensional manifold $M^n$. Is $p$ a locally trivial
fibration?

For $n=1$ an affirmative answer to this problem follows from
results of H.~Whitney~\cite{Whitney}.

\begin{conj}[\v{S}\v{c}epin]
A Serre fibration with a locally arcwise connected metric base is
locally trivial if every fiber of this fibration is homeomorphic
to a given manifold $M^n$ of dimension $n \le 4$.
\end{conj}

In dimension $n=1$ \v{S}\v{c}epin's Conjecture is valid even for
non-compact fibers~\cite{RSS}. \v{S}\v{c}epin also proved that the positive
solution of the Conjecture in dimension $n$ implies positive
solutions of both the CE-problem and the Homeomorphism Group problem in
dimension $n$~\cite{S, DS}. Since the $CE$-problem was solved in 
negative by A.N.~Dranishnikov, it follows that there are dimensional
restrictions in \v{S}\v{c}epin's Conjecture.

The first step toward proving \v{S}\v{c}epin's Conjecture in
dimension $n=2$ was made in~\cite{B}, where existence of local
sections of the fibration was proved under the assumption that the base space is an
$ANR$. We improve this result in two directions. First we prove that any
section of the fibration over closed subset $A$ of the base space can be
extended to a section over some neighborhood of $A$. Secondly, we prove
a theorem on global sections of the fibration.

\begin{2loc}
Let $p\colon E\to B$ be a Serre fibration of locally connected
compacta with all fibers homeomorphic to given
two-dimensional manifold. If $B\in ANR$, then any section of $p$
over a closed subset $A\subset B$ can be extended to a section of
$p$ over a neighborhood of $A$.
\end{2loc}

\begin{2glob}
Let $p\colon E\to B$ be a Serre fibration of locally connected
compactum $E$ onto an $ANR$-compactum $B$ with all fibers
homeomorphic to given two-dimensional manifold $M$. If $M$ is
not homeomorphic to the sphere or the projective plane, then $p$
admits a global section provided one of the following conditions
holds:
\begin{itemize}
\item[(a)] $\pi_1(M)$ is abelian and $H^2(B;\pi_1(F_b))=0$
\item[(b)] $\pi_1(M)$ is non-abelian, $M$ is not homeomorphic to the Klein bottle
   and $\pi_1(B)=0$.
\item[(c)] $M$ is homeomorphic to the Klein bottle and $\pi_1(B)=\pi_2(B)=0$.
\end{itemize}
\end{2glob}

Our strategy of constructing a section of a Serre fibration is as
follows. We consider the inverse (multivalued) mapping and find
its compact submapping admitting continuous approximations. Then
we take very close continuous approximation and use it to find
again a compact submapping with small diameters of fibers
admitting continuous approximations. When we continue this process
we get a sequence of compact submappings with diameters of fibers
tending to zero. This sequence will converge to the desired
singlevalued submapping (selection).

Section~\ref{S:approx} is devoted to continuous approximations of
multivalued maps. We prove filtered finite dimensional
approximation theorem (Theorem~\ref{thmappr}) and then apply it in
a usual way (compare with~\cite{GGK}) to prove an approximation theorem
for maps of ANR-spaces. Since we are going to use singular
filtrations of multivalued maps instead of usual filtrations, our
Theorem~\ref{thmappr} generalizes the Filtered approximation
theorem proved in~\cite{SB}. But the proof of our singular version
of filtered approximation theorem in full generality requires a
lot of technical details. Consequently we decided to present only
the version that we need -- for compact maps of metric
spaces.

Let us recall some definitions and introduce our notation.
All spaces will be separable metrizable.
By a mapping we understand a continuous single-valued mapping.
We equip the product $X\times Y$ with the metric
$${\rm dist}_{X\times Y}((x,y),(x',y'))={\rm dist}_X(x,x')+{\rm dist}_Y(y,y').$$
By $O(x,\varepsilon)$ we denote the open $\varepsilon$-neighborhood of the point $x$.

A multivalued mapping $F\colon X\to Y$ is called a {\it submapping}
(or {\it selection}) of a multivalued mapping $G\colon X\to Y$ if
$F(x)\subset G(x)$ for every $x\in X$. The {\it gauge} of a
multivalued mapping $F\colon X\to Y$ is defined as ${\rm cal}
(F)=\sup\{{\rm diam} F(x)\mid x\in X\}$. The {\it graph} of
multivalued mapping $F\colon X\to Y$ is the subset
$\Gamma_F=\{(x,y)\in X\times Y\mid y\in F(x)\}$ of the product
$X\times Y$. For arbitrary subset $\mathcal U\subset X\times Y$
denote by $\mathcal U(x)$ the subset $pr_Y(\mathcal
U\cap(\{x\}\times Y))$ of $Y$. Then for the graph $\Gamma_F$ we
have $\Gamma_F(x)=F(x)$.

A multivalued mapping $G\colon X\to Y$ is called
{\it complete} if all sets $\{x\}\times G(x)$ are closed
with respect to some $G_\delta$-set $S\subset X\times Y$
containing the graph of this mapping.
A multivalued mapping $F\colon X\to Y$ is called
{\it upper semicontinuous} if for any open set
$U\subset Y$ the set $\{x\in X\mid F(x)\subset U\}$ is open in $X$.
A {\it compact} mapping is an upper semicontinuous
multivalued mapping with compact images of points.

An increasing sequence (finite or countably infinite) of subspaces
$$ Z_0\subset Z_1\subset Z_2\subset\dots\subset Z $$
is called a {\it  filtration}  of space $Z$. A sequence of
multivalued mappings $\{F_k\colon X\to Y\}$ is called a {\it
filtration of multivalued mapping} $F\colon X\to Y$ if for any
$x\in X$, $\{F_k(x)\}$ is a filtration of $F(x)$.

We say that a filtration of multivalued mappings $G_i\colon X\to Y$
is {\it complete} (resp. {\it compact}) if every mapping $G_i$
is complete (resp. compact).

\section{Local properties of multivalued mappings}\label{S:local}

Let $\gamma$ be a property of a topological space such that every
open subspace inherits this property: if a space $X$ satisfies
$\gamma$, then any open subspace $U\subset X$ also satisfies
$\gamma$. We say that a space $Z$ satisfies $\gamma$ {\it locally}
if every point $z\in Z$ has a neighborhood with this property.

For a multivalued map $F\colon X\to Y$ to satisfy $\gamma$ locally
we not only require that every point-image $F(x)$ has this property locally,
but for any points $x\in X$ and $y\in F(x)$
there exist a neighborhood $W$ of $y$ in $Y$ and $U$ of $x$ in $X$
such that $W\cap F(x')$ satisfies $\gamma$ for every point $x'\in U$.
And we use the word "equi" for local properties of multivalued maps.

The first example of such a property is the local compactness.

\begin{defin}
A space $X$ is called {\it locally compact} if every point $x\in X$
has a compact neighborhood. We say that a multivalued map $F\colon X\to Y$ is {\it equi locally compact}
if for any points $x\in X$ and $y\in F(x)$
there exists a neighborhood $W$ of $y$ in $Y$ and $U$ of $x$ in $X$
such that $W\cap F(x')$ is compact for every point $x'\in U$.
\end{defin}

Another local property we are going to use is the hereditary
asphericity. Recall that a compactum $K$ is called {\it
approximately aspherical} if for any (equivalently, for some)
embedding of $K$ into an ANR-space $Y$ every neighborhood $U$ of
$K$ in $Y$ contains a neighborhood $V$ with the following
property: any mapping of the sphere $S^n$ into $V$ is
homotopically trivial in $U$ provided $n\ge 2$.

\begin{defin}
We call a space $Z$ {\it hereditarily aspherical}
if any compactum $K\subset Z$ is approximately aspherical.

A space $Z$ is said to be {\it locally hereditarily aspherical}
if any point $z\in Z$ has a hereditarily aspheric neighborhood.
\end{defin}

It is easy to prove that the 2-dimensional Euclidean space is
hereditary aspherical. Note that any 2-dimensional manifold is
locally hereditarily aspherical.

\begin{defin}
We say that a multivalued map $F\colon X\to Y$ is {\it equi
locally hereditarily aspherical} if for any points $x\in X$ and
$y\in F(x)$ there exists a neighborhood $W$ of $y$ in $Y$ and $U$
of $x$ in $X$ such that $W\cap F(x')$ is hereditarily aspherical for
every point $x'\in U$.
\end{defin}

Now we consider different properties of pairs of spaces and define the
corresponding local properties for spaces and multivalued maps.
We follow definitions and notations from~\cite{DM}.

\begin{defin}
An ordering $\alpha$ of subsets of a space $Y$ is {\it proper} provided:
\begin{itemize}
\item[(a)] If $W\alpha V$, then $W\subset V$;
\item[(b)] If $W\subset V$, and $V\alpha R$, then $W\alpha R$;
\item[(c)] If $W\alpha V$, and $V\subset R$, then $W\alpha R$.
\end{itemize}
\end{defin}

\begin{defin}
Let $\alpha$ be a proper ordering.
\begin{itemize}
\item[(a)]
A space $Y$ is {\it locally of type $\alpha$} if, whenever
$y\in Y$ and $V$ is a neighbourhood of $y$, then there a neighbourhood
$W$ of $y$ such that $W\alpha V$.
\item[(b)]
A multivalued mapping $F\colon X\to Y$ is {\it lower $\alpha$-continuous}
if for any points $x\in X$ and $y\in F(x)$ and for any neighbourhood $V$
of $y$ in $Y$ there exist neighbourhoods $W$ of $y$ in $Y$
and $U$ of $x$ in $X$ such that $(W\cap F(x'))\alpha(V\cap F(x'))$
provided $x'\in U$.
\end{itemize}
\end{defin}

For example, if $W\alpha V$ means that $W$ is contractible in
$V$, then locally of type $\alpha$ means locally contractible.
Another topological property which arises in this manner
is $LC^n$ (where $W\alpha V$ means that every continuous mapping
of the $n$-sphere into $W$ is homotopic to a constant mapping in $V$)
and the corresponding lower $\alpha$-continuity of multivalued map
is called {\it lower $(n+1)$-continuity}.
For the special case $n=-1$ the property $W\alpha V$ means that $V$ is
non-empty, and lower $\alpha$-continuity is the {\it lower semicontinuity}.

The following result is weaker than Lemma~3.5 from~\cite{BCK}. We
will use it with different properties $\alpha$ in
Section~\ref{33}.

\begin{lem}\label{L32}
Let a lower $\alpha$-continuous mapping $\Phi\colon X\to Y$ of
compactum $X$ to a metric space $Y$ contains a compact submapping
$F$. Then for any $\varepsilon>0$ there exists a positive number
$\delta$ such that for every point $(x,y)\in O(\Gamma_F,\delta)$
we have $(O(y,\delta)\cap\Phi(x))\ \alpha\
(O(y,\varepsilon)\cap\Phi(x))$.
\end{lem}

In order to use results from~\cite{SB} we need a local property
called polyhedral $n$-connectedness.
A pair of spaces $V\subset U$ is called {\it polyhedrally $n$-connected}
if for any finite $n$-dimensional polyhedron $M$ and its
closed subpolyhedron $A$ any mapping of $A$ into $V$ can be extended
to a map of $M$ into $U$.
Note that for spaces being locally polyhedrally $n$-connected
is equivalent to be $LC^{n-1}$ (it follows from Lemma~\ref{lempolcont}).
The corresponding local property of multivalued maps
is called {\it polyhedral lower $n$-continuity}.

\begin{lem} \label{lempolcont} 
Any lower $n$-continuous multivalued mapping is
lower polyhedrally $n$-continuous.
\end{lem}

\begin{proof}
The proof easely follows from the fact that in a connected filtration
$Z_0\subset Z_1\subset \dots\subset Z_n$ of spaces the pair $Z_0\subset Z_n$
is polyhedrally $n$-connected.
Given a mapping $f\colon A\to Z_0$ of subpolyhedron $A$ of $n$-dimensional
polyhedron $P$, we extend it successively over skeleta $P^{(k)}$
of $P$ such that the image of $k$-dimensional skeleton $P^{(k)}$
is contained in $Z_k$. Resulting map gives us an extension
$\widetilde f\colon P\to Z_n$ of $f$ which proves that the pair $Z_0\subset Z_n$
is polyhedrally $n$-connected.
\end{proof}

A filtration of multivalued maps $\{F_i\}$ is called
{\it polyhedrally connected} if every pair $F_{i-1}(x)\subset F_{i}(x)$
is polyhedrally $i$-connected.
A filtration $\{F_i\}$ is called {\it lower continuous}
if for any $i$ the mapping $F_i$ is lower $i$-continuous.

\begin{lem} \label{lemma2fibrprop}
If $p\colon E\to B$ is a Serre fibration of $LC^0$-compacta
with fibers homeomorphic to a given 2-dimensional compact manifold, then
the multivalued mapping $p^{-1}\colon B\to E$ is
\begin{itemize}
\item equi locally hereditarily aspherical
\item polyhedrally lower $2$-continuous
\end{itemize}
\end{lem}

\begin{proof}
Since every open proper subset of a two-dimensional manifold is aspherical,
every compact proper subset of 2-manifold is approximately aspherical.
Therefore, the mapping $F$ is equi locally hereditarily aspherical.

It follows from a theorem of McAuley \cite{McAT} that the mapping $p^{-1}$ is
lower 2-continuous.
By Lemma~\ref{lempolcont}, the mapping $p^{-1}$ is
polyhedrally lower 2-continuous.
\end{proof}

\section{Singlevalued approximations}\label{S:approx}

\begin{defin}
A {\it singular pair} of spaces is a triple $(Z,\phi,Z')$
where $\phi\colon Z\to Z'$ is a mapping.

We say that a space $Z$ contains a {\it singular filtration
of spaces} if a finite sequence of pairs $\{(Z_i,\phi_i)\}_{i=0}^n$
is given where $Z_i$ is a space and $\phi_i\colon Z_i\to Z_{i+1}$ is a map
(we identify $Z_{n+1}$ with $Z$).
\end{defin}

For a multivalued map $F\colon X\to Y$ it is useful to consider its
{\it graph fibers} $\{x\}\times F(x)\subset \Gamma_F$
instead of usual fibers $F(x)\subset Y$.
While the graph fibers are always homeomorphic to the usual fibers,
different graph fibers do not intersect (the usual fibers
may intersect in $Y$). We denote the graph fiber of the map $F$
over a point $x\in X$ by $F^\Gamma(x)$.

To define the notion of singular filtration for multivalued maps
we introduce a notion of fiberwise transformation of multivalued maps.

\begin{defin}
For multivalued mappings $F$ and $G$ of a space $X$ a
{\it fiberwise transformation from $F$ to $G$}
is a continuous mapping $T\colon \Gamma_F\to \Gamma_G$
such that $T(F^\Gamma(x))\subset G^\Gamma(x)$ for every $x\in X$.

A {\it fiber} $T(x)$ of the fiberwise transformation $T$ over
the point $x\in X$ is a mapping $T(x)\colon F(x)\to G(x)$
determined by $T$.

We say that a multivalued mapping $F\colon X\to Y$ contains a
{\it singular filtration of multivalued maps}
if a finite sequence of pairs $\{(F_i,T_i)\}_{i=0}^n$
is given where $F_i\colon X\to Y_i$ is a multivalued mapping
and $T_i$ is a fiberwise transformation from $F_i$ to $F_{i+1}$
(we identify $F_{n+1}$ with $F$).
\end{defin}

To construct continuous approximations of multivalued maps we
need the notion of approximate asphericity.

\begin{defin}
A pair of compacta $K\subset K'$ is called
{\it approximately $n$-aspherical} if for any embedding of
$K'$ into $ANR$-space $Z$
for every neighborhood $U$ of $K'$ in $Z$ there exists
a neighborhood $V$ of $K$ such that any mapping
$f\colon S^n\to V$ is homotopically trivial in $U$.

A compactum $K$ is {\it approximately $n$-aspherical} if the pair
$K\subset K$ is approximately $n$-aspherical.
\end{defin}

The following is a singular version of approximate asphericity.

\begin{defin}
A singular pair of compacta $(K,\phi,K')$ is called
{\it approximately $n$-aspherical} if for any embeddings
$K\subset Z$ and $K'\subset Z'$ in $ANR$-spaces and for
any extension of $\phi$ to a map $\widetilde \phi\colon OK\to Z'$
of some neighborhood $OK$ of $K$ the following holds:
for every neighborhood $U$ of $K'$ in $Z'$ there exists
a neighborhood $V$ of $K$ in $OK$ such that for any mapping
$f\colon S^n\to V$ the spheroid $\widetilde\phi\circ f\colon S^n\to U$
is homotopically trivial in $U$.
\end{defin}

Following R.C.~Lacher~\cite{L}, one can prove that this notion does
not depend neither on the choices of $ANR$-spaces $Z$ and $Z'$
nor on the embeddings of $K$ and $K'$ into these spaces.

\begin{defin}
A singular filtration of compacta $\{(K_i,\phi_i)\}_{i=0}^n$
is called {\it approximately connected} if for every $i<n$
the singular pair $(K_i,\phi_i,K_{i+1})$ is approximately $i$-aspherical.
\end{defin}

Clearly, a singular pair of compacta $(K,\phi,K')$ is approximately
$n$-aspherical in either of the following three situations:
compactum $K$, compactum $K'$, or the pair $\phi(K)\subset K'$
is approximately $n$-aspherical.

\begin{defin}\label{appcon}
A singular filtration $\mathcal F=\{(F_i,T_i)\}_{i=0}^n$ of compact
mappings $F_i\colon X\to Y_i$ is said to be {\it approximately connected}
if for every point $x\in X$ the singular filtration of compacta
$\{(F_i(x),T_i(x))\}_{i=0}^n$ is approximately connected.

An approximately connected singular filtration
$\mathcal F=\{(F_i\colon X\to Y_i,T_i)\}_{i=0}^n$
is said to be {\it approximately $\infty$-connected}
if the mapping $F_n$ has approximately $k$-aspherical point-images
$F_n(x)$ for all $k\ge n$ and all $x\in X$.
\end{defin}

Note that if a singular filtration $\mathcal F=\{(F_i,T_i)\}_{i=0}^n$
is approximately $\infty$-connected, then the mapping $F_n$ contains
an approximately connected singular filtration of any given finite length.

We will reduce our study of singular filtrations to the study
of usual filtrations using the following cylinder construction.

\begin{defin}
For a continuous singlevalued mapping $f\colon X\to Y$ we define a
{\it cylinder} of $f$ denoted by ${\rm cyl} (f)$ as a space
obtained from the disjoint union of $X\times [0,1]$ and $Y$ by
identifying each $\{x\}\times \{1\}$ with $f(x)$.
\end{defin}

Note that the cylinder ${\rm cyl} (f)$ contains a homeomorphic
copy of $Y$ called the {\it bottom} of the cylinder, and a
homeomorphic copy of $X$ as $X\times \{0\}$ called the {\it top}
of the cylinder.

\begin{rem} \label{remretr}
There is a natural deformation retraction $r\colon {\rm cyl}(f)\to
Y$ onto the bottom $Y$. Clearly, the fiber of the mapping $r$ over
a point $y\in Y$ is either one point $\{y\}$ or a cone over the
set $f^{-1}(y)$. Therefore, if the map $f$ is proper, then $r$ is
$UV^\infty$-mapping.
\end{rem}

\begin{rem} \label{rememb}
Suppose that $X$ is embedded into Banach space $B_1$ and $Y$ is
embedded into Banach space $B_2$. Then we can naturally embed the
cylinder ${\rm cyl} (f)$ into the product $B_1\times \mathbb
R\times B_2$. The embedding is clearly defined on the top as
embedding into $B_1\times \{0\}\times\{0\}$ and on the bottom as
embedding into $\{0\}\times \{1\}\times B_2$. We extend these
embeddings to the whole cylinder by sending its point $\{x\}\times
\{t\}$ to the point $\{(1-t)\cdot x\}\times t\times \{t\cdot
f(x)\}$.
\end{rem}

\begin{lem} \label{lemasphcomp}
If a singular pair of compacta $(K,\phi,K')$ is approximately
$n$-aspherical, then the pair $K\subset {\rm cyl}(\phi)$ is
approximately $n$-aspherical.
\end{lem}

\begin{proof}
Let us fix embeddings of $K$ into Banach space $B_1$, of $K'$ into
Banach space $B_2$, and of the cylinder ${\rm cyl}(\phi)$ into the
product $B=B_1\times \mathbb R\times B_2$ as described in
Remark~\ref{rememb}. Fix a neighborhood $U$ of ${\rm cyl}(\phi)$
in $B$. Extend the mapping $\phi$ to a map $\phi_1\colon B_1\to
B_2$. Take a neighborhood $V_1$ of the top of our cylinder in
$B_1$ such that the cylinder ${\rm cyl}(\phi_1|_{V_1})$ is
contained in $U$. Using approximate $n$-asphericity of the pair
$(K,\phi,K')$ we find for a neighborhood $U\cap
\{0\}\times\{1\}\times B_2$ of $K'$ in $\{0\}\times\{1\}\times
B_2$ a neighborhood $V'$ of $K$ in $B_1\times\{0\}\times\{0\}$.
Let $\varepsilon$ be a positive number such that the product
$V=V'\times (-\varepsilon,\varepsilon)\times O(0,\varepsilon)$ is
contained in $U$.

Given a spheroid $f\colon S^n\to V$ we retract it into
$V'\times\{0\}\times\{0\}$, then retract it to the bottom of the
cylinder ${\rm cyl}(\phi_1|_{V_1})$ using Remark~\ref{remretr},
and finally contract it to a point inside $U\cap
\{0\}\times\{1\}\times B_2$. Clearly, the whole retraction sits
inside $U$, as required.
\end{proof}

\begin{defin}
Let $\mathcal F=\{(F_i\colon X\to Y_i,T_i)\}_{i=0}^n$ be a
singular filtration of a multivalued mapping $F\colon X\to
Y=Y_{n+1}$. If all the spaces $Y_i$ are Banach, then for a
multivalued mapping $\mathbb F$ from $X$ to $\mathbb Y=Y\times
\prod_{i=0}^n (Y_i\times\mathbb R)$ defined as $\mathbb
F(x)=\cup_{k=0}^n{\rm cyl}(T_k(x))$ we can define a {\it cylinder}
${\rm cyl}(\mathcal F)$ as a filtration of multivalued maps
$\{\mathbb F_i\}_{i=0}^n$ defined as follows:
\[  \mathbb F_0=F_0 \text{ \ and \ }
    \mathbb F_i(x)=\bigcup_{k=0}^{i-1}{\rm cyl}(T_k(x)).  \]
\end{defin}

It is easy to see that for a singular filtration $\mathcal
F=\{(F_i,T_i)\}_{i=0}^n$ of compact mappings $F_i$ the filtration
${\rm cyl}(\mathcal F)$ consists of compact mappings $\mathcal
F_i$.

\begin{lem} \label{lemapprfil}
If a singular filtration $\mathcal F=\{(F_i,T_i)\}_{i=0}^n$ of
compact maps is approximately connected, then the filtration ${\rm
cyl}(\mathcal F)=\{\mathbb F_i\}_{i=0}^n$ is approximately
connected.
\end{lem}

\begin{proof}
Using Remark~\ref{remretr} it is easy to define a deformation
retraction $r\colon \mathbb F_i(x)\to F_i(x)$ which is
$UV^\infty$-mapping. This retraction defines $UV^\infty$-mapping
of pairs $(\mathbb F_i(x),\mathbb F_{i+1}(x))\to (F_i(x),{\rm
cyl}(T_i(x)))$. By Lemma~\ref{lemasphcomp} the pair
$F_i(x)\subset{\rm cyl}(T_i(x))$ is approximately $i$-aspherical,
and by Pairs Mapping Lemma from~\cite{SB} the pair $\mathbb
F_i(x)\subset\mathbb F_{i+1}(x)$ is also approximately
$i$-aspherical.
\end{proof}

\begin{thm} \label{thmappr}
Let $H\colon X\to Y$ be a multivalued mapping of metric space $X$
to a Banach space $Y$.
If $\dim X\le n$ and $H$ contains approximately connected
singular filtration $\mathcal H=\{(H_i\colon X\to Y_i,T_i)\}_{i=0}^n$
of compact mappings, then any neighborhood $\mathcal U$
of the graph $\Gamma_H$ contains the graph of a single-valued
and continuous mapping $h\colon X\to Y$.
\end{thm}

\begin{proof}
Without loss of generality we may assume that all spaces $Y_i$ are
Banach spaces. We consider $Y$ as a subspace of the product
$\mathbb Y=Y\times \prod_{i=0}^n (Y_i\times\mathbb R)$. Clearly,
$H$ is a submapping of a multivalued mapping $\mathbb H\colon X\to
\mathbb Y$ defined as $\mathbb H(x)=\cup_{k=0}^n{\rm cyl}(T_k(x))$
and $\Gamma_{\mathbb H}$ admits a deformation retraction $R$ onto
$\Gamma_H$. Fix a neighborhood $\mathcal U$ of the graph
$\Gamma_H$ in $X\times Y$. Since all maps $H_i$ are compact,
$\mathbb H$ is also compact and the graph $\Gamma_{\mathbb H}$ is
closed in $X\times\mathbb Y$. Extend the mapping ${\rm pr}_Y\circ
R\colon \Gamma_{\mathbb H}\to Y$ to some neighborhood $\mathcal W$
of $\Gamma_{\mathbb H}$ in $X\times\mathbb Y$ and denote by $R'$
the map of $\mathcal W$ to $X\times Y$ such that ${\rm pr}_Y\circ
R'$ is our extension. Clearly, we may assume that $R'(\mathcal W)$
is contained in $\mathcal U$.

By Lemma~\ref{lemapprfil} the multivalued map $\mathbb H$ admits
approximately connected filtration ${\rm cyl}(\mathcal H)$ of
compact multivalued maps. By Single-Valued Approximation Theorem
from~\cite{SB} there exists a singlevalued continuous mapping
${\bf h}\colon X\to \mathbb Y$ with $\Gamma_{\bf h}\subset\mathcal
W$. Define a singlevalued continuous map $h$ by the equality
$\Gamma_h=R'(\Gamma_{\bf h})$. Clearly, $\Gamma_h$ is contained in
$R'(\mathcal W)\subset \mathcal U$.
\end{proof}

\begin{thm} \label{thmapprox}
Suppose that a compact mapping of separable metric ANRs $F\colon X\to Y$
admits a compact singular approximately $\infty$-connected filtration.
Then for any compact space $K\subset X$ every neighborhood of
the graph $\Gamma_F(K)$ contains the graph of a single-valued
and continuous mapping $f\colon K\to Y$.
\end{thm}

\begin{proof}
Let $\mathcal U$ be an open neighborhood of the graph $\Gamma_F(K)$
in the product $X\times Y$. Since $F$ is upper semicontinuous,
there is a neighborhood $OK$ of compactum $K$ such that
$\Gamma_F(OK)$ is contained in $\mathcal U$.
Since any open subset of separable ANR-space is separable ANR-space
\cite{Han}, we can denote $OK$ by $X$ and consider $\mathcal U$
as an open neighborhood of the graph $\Gamma_F$.

For every point $x\in X$ take open neighborhoods $O_x\subset X$
of the point $x$ and $V_x\subset X$ of the compactum $F(x)$
such that the product $O_x\times V_x$ is contained in $\mathcal U$.
Using upper semicontinuity of $F$ we can choose $O_x$
so small that the following inclusion holds: $F(O_x)\subset V_x$.
Fix an open covering $\omega_1$ of the space $X$ which is
starlike refinement of $\{O_x\}_{x\in X}$.
Let $\omega_2$ be a locally finite open covering of the space $X$
which is starlike refined into $\omega_1$.

There exist a locally finite simplicial complex $L$ and mappings
$r\colon X\to L$ and $j\colon L\to X$ such that the map $j\circ r$
is $\omega_2$-close to ${\rm id}_X$ \cite{Han}. Fix a finite
subcomplex $N\subset L$ containing the compact set $r(K)$. Define
a compact mapping $\Psi\colon N\to Y$ by the formula $\Psi=F\circ
j$. Clearly, the mapping $\Psi$ admits a compact approximately
connected singular filtration of any length (particularly, of the
length $\dim N$). Let us define a neighborhood $\mathcal W$ of the
graph $\Gamma_\Psi$. For every point $q\in N$ we put

$$ \mathcal W(q)=\bigcap\{\mathcal U(y)\mid y\in{\rm st}_{\omega_1}({\rm St}_{\omega_2}(j(q)))\}. $$

By ${\rm St}_{\omega_2}(j(q))$, we denote the star of the point
$j(q)$ with respect to the covering $\omega_2$. And by ${\rm
st}(A,\omega)$, we denote the set $\bigcup\{U\in\omega\mid
A\subset U\}$.

By Theorem~\ref{thmappr} there exists a single-valued continuous
mapping $\psi\colon N\to Y$ such that the graph $\Gamma_\psi$ is
contained in $\mathcal W$. Put $f=\psi\circ r\colon X\to Y$. For
any point $x\in K$ we have $\psi(r(x))\in\cap\{\mathcal U(x')\mid
x'\in{\rm St}_{\omega_2}(j\circ r(x))\}$. Since $x\in{\rm
St}_{\omega_2}(j\circ r(x))$, then $\psi(r(x))\in\mathcal U(x)$.
That is, the graph of $f$ is contained in $\mathcal U$.
\end{proof}

\section{Fibrations with 2-manifold fibers}\label{33}

The following Lemma is a weak form of Compact Filtration Lemma
from~\cite{SB}.

\begin{lem}\label{L31}
Any polyhedrally connected lower continuous finite filtration of
complete mappings of a compact space contains a compact
approximately connected subfiltration of the same length.
\end{lem}

\begin{lem} \label{lemmashrinking}
Let $F\colon X\to Y$ be equi locally hereditarily aspherical, lower
2-continuous complete multivalued mapping of $ANR$-space $X$ to
Banach space $Y$. Suppose that a compact submapping $\Psi\colon
A\to Y$ of $F|_A$ is defined on a compactum $A\subset X$ and
admits continuous approximations. Then for any $\varepsilon>0$
there exists a neighborhood $OA$ of $A$ and a compact submapping
$\Psi'\colon OA\to Y$ of $F|_{OA}$ such that
$\Gamma_{\Psi'}\subset O(\Gamma_\Psi,\varepsilon)$, $\Psi'$ admits
a compact approximately $\infty$-connected filtration, and ${\rm
cal} \Psi'<\varepsilon$.
\end{lem}

\begin{proof}
Fix a positive number $\varepsilon$.
Apply Lemma~\ref{L32} with $\alpha$ being equi local hereditary asphericity
to get a positive number $\varepsilon_2<\varepsilon/4$.
By Lemma~\ref{lempolcont} the mapping $F$ is lower polyhedrally 2-continuous.
Subsequently applying Lemma~\ref{L32} with $\alpha$ being
polyhedral $n$-continuity for $n=2,1,0$, we find positive numbers
$\varepsilon_1$, $\varepsilon_0$, and $\delta$ such that $\delta<\varepsilon_0<\varepsilon_1<\varepsilon_2$
and for every point $(x,y)\in O(\Gamma_{\Psi},\delta)$
the pair $(O(y,\varepsilon_1)\cap F(x),O(y,\varepsilon_2)\cap F(x))$
is polyhedrally 2-connected,
the pair $(O(y,\varepsilon_0)\cap F(x),O(y,\varepsilon_1)\cap F(x))$
is polyhedrally 1-connected, and
the intersection $O(y,\varepsilon_0)\cap F(x)$ is not empty.

Let $f\colon K\to Y$ be a continuous single-valued mapping
whose graph is contained in $O(\Gamma_{\Psi},\delta)$.
Let $f'\colon\mathcal OK\to Y$ be a continuous extension of the mapping $f$
over some neighborhood $\mathcal OK$ such that the graph of $f'$
is contained in $O(\Gamma_{\Psi},\delta)$.
Now we can define a polyhedrally connected filtration
$G_0\subset G_1\subset G_2\colon \mathcal OK\to Y$
of the mapping $F|_{\mathcal OK}$ by the equality
\[
G_i(x)=O(f'(x),\varepsilon_i)\cap F(x).
\]
Since the set $\cup_{x\in\mathcal OK}\{\{x\}\times
O(f'(x),\varepsilon_i)\}$ is open in the product $\mathcal
OK\times Y$ and the mapping $F$ is complete, then $G_i$ is also
complete. Clearly, ${\rm cal} G_2<2\varepsilon_2<\varepsilon$ and
for any point $x\in K$ the set $G_2^\Gamma(x)$ is contained in
$O(\Gamma_\Psi,\varepsilon)$. Now, applying Lemma~\ref{L31} to the
filtration $G_0\subset G_1\subset G_2$, we obtain a compact
approximately connected subfiltration $F_0\subset F_1\subset
F_2\colon \mathcal OK\to Y$. By the choice of $\varepsilon_2$ the
mapping $F_2$ has approximately aspherical point-images. Therefore,
the filtration $F_0\subset F_1\subset F_2$ is approximately
$\infty$-connected. Finally, we put $\Psi'=F_2$.
\end{proof}

\begin{thm} \label{thmapproxsections}
Let $F\colon X\to Y$ be equi locally hereditarily aspherical,
lower 2-continuous complete multivalued mapping
of locally compact $ANR$-space $X$ to Banach space $Y$.
Suppose that a compact submapping $\Psi\colon A\to Y$ of $F|_A$
is defined on compactum $A\subset X$ and admits continuous approximations.
Then for any $\varepsilon>0$ there exists a neighborhood $OA$ of $A$
and a single-valued continuous selection $s\colon OA\to Y$
of $F|_{OA}$ such that $\Gamma_s\subset O(\Gamma_\Psi,\varepsilon)$.
\end{thm}

\begin{proof}
Consider a $G_\delta$-subset $G\subset X\times Y$ such that all
fibers of $F$ are closed in $G$ and fix open sets $G_i\subset
X\times Y$ such that $G=\cap^\infty_{i=1} G_i$. Fix
$\varepsilon>0$ such that $O(\Gamma_\Psi,\varepsilon)\subset G_1$.
By Lemma~\ref{lemmashrinking} there is a neighborhood $U_1$ of $A$
in $X$ and a compact submapping $\Psi_1\colon U_1\to Y$ of
$F|_{U_1}$ such that $\Gamma_{\Psi_1}\subset
O(\Gamma_\Psi,\varepsilon)$, $\Psi_1$ admits a compact
approximately $\infty$-connected filtration, and ${\rm cal}
\Psi_1<\varepsilon$. Since $X$ is locally compact and $A$ is
compact, there exists a compact neighborhood $OA$ of $A$ such that
$OA\subset U_1$. By Theorem~\ref{thmapprox} the mapping
$\Psi_1|_{OA}$ admits continuous approximations. Take
$\varepsilon_1<\varepsilon$ such that the neighborhood $\mathcal
U_1=O(\Gamma_{\Psi_1}(OA),\varepsilon_1)$ lies in
$O(\Gamma_\Psi,\varepsilon)$. Clearly, $\mathcal U_1\subset G_1$.

Now by induction with the use of Lemma~\ref{lemmashrinking}, we
construct a sequence of neighborhoods $U_1\supset U_2\supset
U_3\supset\dots$ of the compactum $OA$, a sequence of compact
submappings $\{\Psi_k\colon U_k\to Y\}^\infty_{k=1}$ of the
mapping $F$, and a sequence of neighborhoods $\mathcal
U_k=O(\Gamma_{\Psi_1}(OA),\varepsilon_k)$ such that for every
$k\ge 2$ we have ${\rm cal}
\Psi_k<\varepsilon_{k-1}/2<\varepsilon/2^k$, and $\mathcal
U_k(OA)$ is contained in $\mathcal U_{k-1}(OA)\cap G_k$. It is not
difficult to choose the neighborhood $\mathcal U_k$ of the graph
$\Gamma_{\Psi_k}$ in such a way that for every point $x\in U_k$
the set $\mathcal U_k(x)$ has diameter less than $\frac{3}{2^k}$.

Then for any $m\ge k\ge 1$ and for any point $x\in OA$ we have
$\Psi_m(x)\subset O(\Psi_k(x),\frac{3}{2^k})$; this implies the
fact that the sequence $\{\Psi_k|_{OA}\}_{k=1}^\infty$ is a Cauchy
sequence. Since $Y$ is complete, there exists the limit $s\colon
OA\to Y$ of this sequence. The mapping $s$ is single-valued by the
condition ${\rm cal}\Psi_k<\frac{1}{2^k}$ and is upper
semicontinuous (and, therefore, is continuous) by the upper
semicontinuity of all the mappings $\Psi_k$. Clearly, for any
$x\in OA$ the point $s(x)$ lies in $G(x)$ and is a limit point of
the set $F(x)$. Since $F(x)$ is closed in $G(x)$, then $s(x)\in
F(x)$, i.e. $s$ is a selection of the mapping $F$.
\end{proof}

\begin{thm}
Let $p\colon E\to B$ be a Serre fibration of locally connected
compacta with all fibers homeomorphic to some fixed
two-dimensional manifold. If $B\in ANR$, then any section of $p$
over closed subset $A\subset B$ can be extended to a section of
$p$ over some neighborhood of $A$.
\end{thm}

\begin{proof}
Let $s\colon A\to E$ be a section of $p$ over $A$.
Embed $E$ into Hilbert space $l_2$ and consider a multivalued mapping
$F\colon B\to l_2$ defined as follows:
$$
   F(b)=\begin{cases} s(b), &\text{if \;$b\in A$}\\
               p^{-1}(b),   &\text{if \;$x\in B\setminus
               A$.}\end{cases}
$$
Since every fiber $p^{-1}(b)$ is compact, the mapping $F$ is complete.
By Lemma~\ref{lemma2fibrprop} the mapping $F$ is equi
locally hereditarily aspherical and lower $2$-continuous.
We can apply Theorem~\ref{thmapproxsections} to the mapping $F$
and its submapping $s$ to find a single-valued continuous selection
$\widetilde s\colon OA\to l_2$ of $F|_{OA}$.
By definition of $F$, we have $\widetilde s|_{A}=F|_{A}=s$.
Clearly, $\widetilde s$ defines a section of the fibration $p$
over $OA$ extending $s$.
\end{proof}

\begin{defin}
For a mapping $p\colon E\to B$ we say that $s\colon B\to E$ is
{\it $\varepsilon$-section} if the map $p\circ s$ is
$\varepsilon$-close to the identity ${\rm id}_B$.
\end{defin}

The following proposition easily follows from Theorem~4.1 of the paper~\cite{Mi}.

\begin{pro} \label{profibr}
If $p\colon E\to B$ is a locally trivial fibration of finite-dimensional
compacta with locally contractible fiber, then there is $\varepsilon>0$ such that an
existence of $\varepsilon$-section for $p$ implies an existence of a section for $p$.
\end{pro}

We will use the following two propositions in the proof of
existence of global sections in Serre fibrations. For the
definition and basic properties of Menger manifolds we refer the
reader to~\cite{bestvina}. Proofs of these two propositions
follow from Bestvina's construction of Menger
manifold~\cite{bestvina} and Dranishnikov's triangulation theorem
for Menger manifolds~\cite{Dr2}.

\begin{pro} \label{propoly}
Let $X$ be a compact 2-dimensional Menger manifold.
For any $\varepsilon>0$ there exist a finite polyhedron $P$ and maps $g\colon X\to P$
and $h\colon P\to X$ such that $h\circ g$ is $\varepsilon$-close to the identity.
If $\pi_1(X)=0$, then we may choose $P$ with $\pi_1(P)=0$.
\end{pro}

\begin{pro} \label{pro3poly}
Let $X$ be a compact 3-dimensional Menger manifold with $\pi_1(X)=\pi_2(X)=0$.
For any $\varepsilon>0$ there exist a finite polyhedron $P$ with
$\pi_1(P)=\pi_2(P)=0$ and maps $g\colon X\to P$
and $h\colon P\to X$ such that $h\circ g$ is $\varepsilon$-close to the identity.
\end{pro}

\begin{thm} \label{glob2sec}
Let $p\colon E\to B$ be a Serre fibration of locally connected
compactum $E$ onto an $ANR$-compactum $B$ with all fibers
homeomorphic to a given two-dimensional manifold $M$. If $M$ is
not homeomorphic to the sphere or the projective plane, then $p$
admits a global section provided one of the following conditions
holds:
\begin{itemize}
\item[(a)] $\pi_1(M)$ is abelian and $H^2(B;\pi_1(F_b))=0$
\item[(b)] $\pi_1(M)$ is non-abelian, $M$ is not homeomorphic to the Klein bottle
   and $\pi_1(B)=0$.
\item[(c)] $M$ is homeomorphic to the Klein bottle and $\pi_1(B)=\pi_2(B)=0$.
\end{itemize}
\end{thm}

\begin{proof}
Embed $E$ into the Hilbert space $l_2$ and consider a multivalued mapping
$F\colon B\to l_2$ defined as $F=p^{-1}$.
Since every fiber $p^{-1}(b)$ is compact, the mapping $F$ is complete.
It follows from Lemma~\ref{lemma2fibrprop} that the mapping $F$ is
equi locally hereditarily aspherical and lower $2$-continuous.

Now we show that $F$ admits a compact singular approximately
$\infty$-connected filtration.
In cases (a) and (b) there exists $UV^1$-mapping $\mu$
of Menger 2-dimensional manifold $L$
onto $B$~\cite{Dr1}. Note that $\pi_1(L)=0$ if $\pi_1(B)=0$.
In case (c) we consider $UV^2$-mapping $\mu$
of Menger 3-dimensional manifold $L$ onto $B$~\cite{Dr1};
note that $\pi_1(L)=\pi_2(L)=0$ if $\pi_1(B)=\pi_2(B)=0$.
Since $\dim L<\infty$, the induced fibration $p_L=\mu^*(p)\colon E_L\to L$
is locally trivial~\cite{HD}.
By Proposition~\ref{profibr} there is $\varepsilon>0$ such that an
existence of $\varepsilon$-section for $p_L$ implies an
existence of a section for $p_L$.
In cases (a) and (b), by Proposition~\ref{propoly}, there exist a 2-dimensional
finite polyhedron $P$ and continuous maps $g\colon L\to P$
and $h\colon P\to L$ such that $h\circ g$ is $\varepsilon$-close
to the identity (we assume $\pi_1(P)=0$ in case $\pi_1(B)=0$).
In case (c) by Proposition~\ref{pro3poly} there exist a 3-dimensional
finite polyhedron $P$ with $\pi_1(P)=\pi_2(P)=0$
and continuous maps $g\colon L\to P$ and $h\colon P\to L$
such that $h\circ g$ is $\varepsilon$-close to the identity.

Consider a locally trivial fibration $p_P=h^*(p_L)\colon E_P\to P$.

\noindent {\bf Claim.} \ {\it The fibration $p_P$ has a section
$s_P$.}
\begin{proof}
(a) If $\pi_1(M)$ is abelian and
$H^2(B;\pi_1(F_b))=0=H^2(P;\pi_1(F_b))$, the fibration $p_P$ has a
section $s_P$~\cite{W}.

(b) Since $\pi_1(P)=0$ and $\dim P=2$, then $P$ is homotopy
equivalent to a bouquet of 2-spheres $\Omega=\vee_{i=1}^m S^2_i$.
Let $\psi\colon P\to\Omega$ and $\phi\colon\Omega\to P$ be maps
such that $\phi\circ \psi$ is homotopic to the identity ${\rm
id}_P$. The locally trivial fibration over a bouquet
$p_\Omega=\phi^*(p_P)\colon E_\Omega\to\Omega$ has a global
section if and only if it has a section over every sphere of the
bouquet. If the fiber $M$ has non-abelian fundamental group and is
not homeomorphic to Klein bottle, then the space of
autohomeomorphisms ${\rm Homeo}(M)$ has simply connected identity
component~\cite{BH} and therefore any locally trivial fibration
over 2-sphere with fiber homeomorphic to $M$ has a section (in
fact, this fibration is trivial). Hence, the fibration $p_\Omega$
has a section $s_\Omega$. This section defines a lifting of the
map $\phi\circ\psi\colon P\to P$ with respect to $p_P$. Since
$p_P$ is a Serre fibration and $\phi\circ\psi$ is homotopic to the
identity, the identity mapping ${\rm id}_P$ has a lifting
$s_P\colon P\to E_P$ with respect to $p_P$ which is simply a
section of $p_P$.

(c) Since $\pi_1(P)=\pi_2(P)=0$ and $\dim P=3$, then $P$ is
homotopy equivalent to a bouquet of 3-spheres $\Omega=\vee_{i=1}^m
S^3_i$. Let $\psi\colon P\to\Omega$ and $\phi\colon\Omega\to P$ be
maps such that $\phi\circ \psi$ is homotopic to the identity ${\rm
id}_P$. The locally trivial fibration over the bouquet
$p_\Omega=\phi^*(p_P)\colon E_\Omega\to\Omega$ has a global
section if and only if it has a section over every sphere of the
bouquet. Since the space of autohomeomorphisms of the Klein bottle
${\rm Homeo}(K^2)$ has $\pi_2({\rm Homeo}(K^2))=0$~\cite{BH}, any
locally trivial fibration over 3-sphere with fiber homeomorphic to
$K^2$ has a section (in fact, this fibration is trivial). Hence,
the fibration $p_\Omega$ has a section $s_\Omega$. This section
defines a lifting of the map $\phi\circ\psi\colon P\to P$ with
respect to $p_P$. Since $p_P$ is a Serre fibration and
$\phi\circ\psi$ is homotopic to the identity, the identity mapping
${\rm id}_P$ has a lifting $s_P\colon P\to E_P$ with respect to
$p_P$ which is simply a section of $p_P$.
\end{proof}

By the construction of $P$ the section $s_P$ defines an
$\varepsilon$-section for $p_L$. Therefore, $p_L$ has a section $s_L$.
Clearly, $s_L$ defines a lifting $T\colon L\to E$
of $\mu$ with respect to $p$.
Finally, we define compact singular filtration
$\mathcal F=\{(F_i,T_i)\}_{i=0}^2$ of $F$ as follows:
\[  F_0=F_1=\mu^{-1}\colon B\to L,\qquad F_2=F,
\qquad T_i={\rm id} \text{ for } i=0  \] and $T_1$ is defined
fiberwise by $T_1(x)=T|_{\mu^{-1}(x)}\colon \mu^{-1}(x)\to F(x)$.
The filtration $\mathcal F$ is approximately connected since for
$i=0,1$ any compactum $F_i(x)$ is $UV^1$. And $\mathcal F$ is
approximately $\infty$-connected since every compactum $F(x)$ is
an aspherical 2-manifold (and therefore is approximately
$n$-aspherical for all $n\ge 2$).

Now we can apply Theorem~\ref{thmapproxsections} to the mapping $F$
to find a single-valued continuous selection $s\colon B\to l_2$ of $F$.
Clearly, $s$ defines a section of the fibration $p$.
\end{proof}

The following Remark explains the appearence of the condition (c)
in Theorem~\ref{glob2sec}.

\begin{rem}
There exists a locally trivial fibration over 2-sphere with fibers
homeomorphic to Klein bottle having no global section.
\end{rem}

\section{Acknowledgements}

Authors wish to express their sincere thanks to
P.~Akhmetiev, R.J.~Daverman, B.~Hajduk and T.~Yagasaki
for helpful discussions during the development of this work.


\begin{thebibliography}{99}

\bibitem{BH} W.~Balcerak, B.~Hajduk,
{\it Homotopy type of automorphism groups of manifolds},
Colloq. Math. {\bf 45} (1981), 1--33.

\bibitem{bestvina} M.~Bestvina,
{\it Characterizing $k$-dimensional universal Menger compacta},
Mem. Amer. Math. Soc. {\bf 71} (380), 1988.

\bibitem{B} N.~Brodsky,
{\it Sections of maps with fibers homeomorphic to a two-dimensional manifold},
Topology Appl. {\bf 120} (2002), 77--83.

\bibitem{BCK} N.~Brodsky, A.~Chigogidze, A.~Karasev,
{\it Approximations and selections of multivalued mappings of
finite-dimensional spaces},
JP Journal of Geometry and Topology {\bf 2} (2002), 29--73.

\bibitem{DD} R.J.~Daverman, A.N.~Dranishnikov,
{\it Cell-like maps and aspherical compacta},
Illinois J. Math. {\bf 40}:1 (1996), 77--90.

\bibitem{Dr1} A.N.~Dranishnikov,
{\it Absolute extensors in dimension $n$ and $n$-soft
mappings increasing the dimension},
Russian Math. Surveys {\bf 39}:5 (1984), 63--111.

\bibitem{Dr2} A.N.~Dranishnikov,
{\it Universal Menger compacta and universal mappings}, (Russian)
Math. USSR Sbornik {\bf 57}:1 (1987), 131--150.

\bibitem{DS}   A.N.~Dranishnikov, E.V.~\v{S}\v{c}epin,
{\it Cell-like mappings. The problem of increase of dimension},
Russian Math. Surveys {\bf 41}:6 (1986), 59--111.

\bibitem{DM} J.~Dugundji, E.~Michael,
{\it On local and uniformly local topological properties},
Proc. Amer. Math. Soc. {\bf 7} (1956), 304--307.


\bibitem{GGK}   L.~Gorniewicz, A.~Granas, W.~Kryszewski,
{\it On the homotopy method in the fixed point index
theory of multi-valued mappings of compact absolute
neighborhood retracts},
J. Math. Anal. Appl. {\bf 161} (1991), 457--473.


\bibitem{HD}   M.E.~Hamstrom, E.~Dyer,
{\it Regular mappings and the space of homeomorphisms on a 2-manifold},
Duke Math. J. {\bf 25} (1958),  521--531.

\bibitem{Han} O.~Hanner,
{\it Some theorems on absolute neighborhood retracts},
Arciv. Mat. {\bf 1} (1951), 389--408.


\bibitem{L} R.C.~Lacher,
{\it Cell-like mappings and their generalizations},
Bull. Amer. Math. Soc. {\bf 83} (1977), 495--552.

\bibitem{McAT} F.~McAuley, P.A.~Tulley,
{\it Fiber spaces and $n$-regularity},
Topology Seminar Wisconsin, Ann. of Math. Studies 60. 1965.

\bibitem{Mi} E.~Michael,
{\it Continuous  Selections, II},
Ann. Math. {\bf 64} (1956), 562--580.

\bibitem{RS} D.~Repovs, P.V.~Semenov,
{\it Continuous selections of multivalued mappings},
Kluwer Academic Publishers, Dordrecht, 1998.

\bibitem{RSS} D.~Repovs, P.V.~Semenov, E.V.~\v{S}\v{c}epin, 
{\it Topologically regular maps with fibers homeomorphic to
a one-dimensional polyhedron},
Houston J. Math. {\bf 23} (1997), 215--230.

\bibitem{S}  E.V.~\v{S}\v{c}epin, 
{\it On homotopically regular mappings of manifolds},
Banach Centre Publ. 18, PWN, Warszawa, 1986, 139--151.

\bibitem{SB} E.V.~\v{S}\v{c}epin, N.~Brodsky,
{\it Selections of filtered multivalued mappings},
Proc. Steklov Inst. Math. {\bf 212} (1996), 218--228.

\bibitem{W} G.W.~Whitehead,
{\it Elements of homotopy theory},
Graduate Texts in Mathematics, 61. Springer-Verlag, New York-Berlin, 1978.

\bibitem{Whitney}  H.~Whitney, 
{\it Regular families of curves},
Ann. Math. (2) {\bf 34} (1933), 244--270.

\end{thebibliography}
\end{document}